\documentclass{amsart}
\usepackage{color}

\newtheorem{theorem}{Theorem}[section]
\newtheorem{lemma}[theorem]{Lemma}

\theoremstyle{definition}

\theoremstyle{remark}
\newtheorem{remark}[theorem]{Remark}
\numberwithin{equation}{section}

\begin{document}

\title[Optimal bounds for self-similar solutions to coagulation equations]{Optimal bounds for self-similar solutions to coagulation equations with
 product kernel}

\author{B. Niethammer}
\address{Mathematical Institute, University of Oxford, 24-29 St.
Giles, Oxford, OX1 3LB, England }
\email{niethammer@maths.ox.ac.uk}
\thanks{This work was supported by the EPSRC Science and Innovation award to the
Oxford Centre for Nonlinear PDE (EP/E035027/1)
and through the DGES Grant MTM2007-61755. The authors also gratefully acknowledge the hospitality
of the Isaac Newton Institute for Mathematical Sciences where part of this work was done during
the program on PDE in Kinetic Theories.
}

\author{J.J.L. Vel\'azquez}
\address{
ICMAT (CSIC-UAM-UC3M-UCM),
C/ Nicol\'as Cabrera 15,
28049 Madrid, Spain}
\email{jj$\mbox{}_{-}$velazquez@icmat.es}




\begin{abstract}
We consider mass-conserving self-similar solutions of Smoluchows-\linebreak ki's coagulation
equation with  product kernel of homogeneity $2\lambda \in (0,1)$.
We establish rigorously that such solutions
exhibit a singular behavior of the form $x^{-(1+2\lambda)}$ as $x \to 0$. This property
had been conjectured, but only weaker results had been available up to now.
\end{abstract}

\maketitle

\newcommand{\nwc}{\newcommand}

\newcommand{\av}{ -\! \! \! \! \! \!  \!\int}
\newcommand{\tav}{ -\! \!  \! \!  \!\int}

\newcommand {\Matrix}[4]
 {  \left( \begin{array}{cc} #1 & #2\\
                              #3 & #4 \end{array} \right)  }

\newcommand {\Fall}[5]
    { #1 = \left \{
           \begin{array}
           {c@ {\quad : \quad} l}
            #2 & #3 \\
            #4 & #5           \end{array}
           \right .    }

\newcommand {\Falldrei}[7]
    { #1 = \left \{
           \begin{array}
           {c@ {\quad : \quad} l}
           #2 & #3 \\
            #4 & #5 \\
            #6 & #7      \end{array}
           \right .    }

\nwc{\R}{\mathbb R}
\nwc{\Z}{\mathbb Z}
\nwc{\N}{\mathbb N}

\newcommand{\ignore}[1]{}

\nwc{\eps}{\varepsilon}
\nwc{\re}{Re\,}

\nwc{\wto}{\rightharpoonup}
\newcommand{\normal}{\cdot \vec n}

\renewcommand{\AE}{{a.e.\ }}
\nwc{\ds}{\displaystyle}
\newcommand {\bedis} {\begin{displaymath}}
\newcommand {\edis} {\end{displaymath}}
\newcommand{\newbeqna} {\renewcommand {\arraystretch} {2}
                        \begin {displaymath} \begin {array}{crcl}}
\newcommand{\neweqna}{\end{array} \end {displaymath}}

\newcommand{\fbeqna}{\renewcommand {\arraystretch} {1.3}
\begin {displaymath}\begin{array}{rcll}}
\newcommand{\feqna}{\end{array}\end{displaymath}}

\newcommand {\beqna} {\begin{eqnarray*}}
\newcommand {\eqna} {\end{eqnarray*}}
\newcommand {\beqn} {\begin{eqnarray}}
\newcommand {\eqn} {\end{eqnarray}}
\newcommand{\x}{X}
\newcommand{\y}{Y}
\newcommand{\z}{Z}
\newcommand{\uu}{U}
\newcommand{\vv}{V}
\newcommand{\h}{H}
\newcommand{\skp}[2]{\langle #1 \, , #2 \rangle}

\newcommand{\xmy}{x{-}y}
\newcommand{\B}{B}
\newcommand{\p}{\varphi}


\section{Introduction}
\label{S.intro}

Smoluchowski's coagulation equation describes
the irreversible aggregation of clusters by binary collisions in a mean-field approximation.
In the following we denote the number density of clusters of size $\xi$ 
at time $t$ by $f(t,\xi)$.
Clusters of size $\xi$ and $\eta$ can  coalesce  to clusters
of size $\xi+\eta$ at a rate given by a rate kernel $K(\xi,\eta)$.
Then the
 dynamics of $f$ are given by
\begin{equation}\label{coag1}
\frac{\partial}{\partial t} f (\xi,t) = \tfrac 1 2 \int_0^\xi \,d\eta\, K(\xi-\eta,\eta)
f(\eta,t) f(\xi{-}\eta,t)
\,-
\,  f(\xi,t) \int_0^{\infty}\,d\eta\, K(\xi,\eta) f(\eta,t)\,.
\end{equation}
In this article we are particularly interested in self-similarity in Smoluchowski's coagulation equation and
thus we   consider  homogeneous kernels. More precisely, we assume 
that $K \in C^1(\R_+^2)$, $K \geq 0$, $K$ is symmetric and is homogeneous of degree $2\lambda \in (0,1)$,
that is
\begin{equation}
\label{kernel0a}
K(ax,ay) = a^{2\lambda} K(x,y)  \qquad \mbox{ for all } x,y \in \R_+\, \mbox{ and some } \lambda \in (0,1/2)\,.
\end{equation}
Next, we assume that the probabilities for coalescence between particles have a certain power law growth
in the sizes of particles. That is, we
 assume that there exists a positive constant $K_0$  such that
\begin{equation}\label{kernel0b}
\begin{split}
K(x,y) &\leq K_0\big( x^{\alpha} y^{\beta} + x^{\beta}y^{\alpha}\big) \qquad \mbox{ for all } x,y \in \R_+^2
\\
0 & <\alpha \leq \beta <1/2, \qquad \alpha + \beta =2\lambda\,
\,.
\end{split}
\end{equation}
We also need a  non-degeneracy assumption that says that a certain number of coalescence of
particles of comparable size take place. We assume 
that there exists a positive constant $k_0$ such that
\begin{equation}\label{kernel0c}
\min_{[1/4,1]\times [1/4,1]} K(x,y) \geq k_0\,.
\end{equation}
The number $1/4$ could be replaced by any number $a \in (0,1)$.

Kernels of this type are denoted as kernels of {\it Class I} in the review paper  \cite{Le1}.
In particular, the so-called  product kernel
\begin{equation}\label{kernel}
K(\xi,\eta) = \xi^{\alpha} \eta^{\beta} + \xi^{\beta} \eta^{\alpha}
\end{equation}
with $0<\alpha \leq \beta$ 
satisfies all the assumptions \eqref{kernel0a}-\eqref{kernel0c}.

It is well-known \cite{LM1} that for the homogeneity $2\lambda \in (0,1)$  the initial
value problem \eqref{coag1} for  data with finite mass is well-posed and 
 the  mass  $ \int_0^{\infty} \xi f(\xi,t) \,d\xi$
is conserved for all times. It has been conjectured for homogeneous kernels that 
solutions of \eqref{coag1} exhibit self-similar form  for large times. However,
only for special kernels such as $K=1$ or $K=x+y$, this hypothesis could be verified.
These kernels have explicit fast decaying self-similar solutions and recently also new
families
of  self-similar solutions  have been discovered \cite{Bertoin1,MP1} that
have algebraic decay and infinite mass.
Furthermore, their domain of attraction 
 under weak convergence
has been completely characterized  \cite{MP1}. 

However, self-similarity is still only poorly understood for non-solvable kernels such
as the ones in \eqref{kernel0a}-\eqref{kernel0c}. In fact, not much is known about
 the structure of self-similar solutions themselves. Physicists \cite{Le1,DE1}
 have derived asymptotics
for small and large clusters under the assumption that a fast decaying sufficiently regular
solution exists.  A rigorous proof of 
 existence of fast decaying
mass-conserving self-similar solutions
for a class of homogeneous
kernels has however only recently been established \cite{EMR,FL1}. As far as we are aware, nothing
is known about self-similar solutions with algebraic decay or  the uniqueness
of mass-conserving self-similar solutions. 
As a further step towards a better understanding of the latter, some effort
has been undertaken to obtain more qualitative information about the self-similar
solutions obtained in \cite{EMR,FL1}. Certain regularity properties and 
estimates
on their precise decay at infinity and their behaviour for small
clusters have been derived in \cite{CM, EM, FL1,FL2}.
It turns out that these results are optimal for the so-called sum kernel, 
that is $K$ as in \eqref{kernel}
 with $\alpha = 0$,  
but they are  only suboptimal for the  product kernel, that is
the case $\alpha>0$. More precisely, in the case $\alpha=0$ self-similar solutions exhibit
a singular power-law behavior of the form $x^{-\tau}$ for some $\tau <1+2\lambda$ that is
determined in a nonlocal way by the $2\lambda$-th moment of the solution itself. For the case
$\alpha>0$  the predicted power-law is $x^{-(1+2\lambda)}$ and thus completely different.
Our contribution in this paper is to establish rigorously the expected singular power-law behavior
for self-similar solutions for kernels satisfying \eqref{kernel0a}-\eqref{kernel0c} in the case $\alpha>0$.
Our method has the advantage of being completely elementary.

 From the
physical point of view $\alpha >0$ means that a given particle is more
likely to interact with particles having comparable sizes than with smaller
ones. On the contrary, in the case $\alpha =0,$ a given particle has similar
probability of interacting with small particles and with comparable ones.
Our results in this paper confirm that in the case $\alpha>0$ the distribution of small 
particles (in self-similar variables) is basically determined by the collisions
with comparable particles, while the analysis in \cite{CM,FL2} for the case $\alpha=0$ shows that
the distribution for small particles is mostly due to the collisions with larger particles.

In order to describe our results in more detail we first derive the equation that is
satisfied by mass-conserving self-similar solutions of \eqref{coag1}. Such solutions are of
the form 
\begin{equation}\label{coag2}
f(\xi,t) = \frac{1}{s^2(t)} g\big( \frac{\xi}{s(t)} \big)
\end{equation}
with an increasing function $s(t)$. 
Using the ansatz \eqref{coag2} in \eqref{coag1} and setting $\xi/s=x$ and $\eta/s=y$
we find that 
 $s$ must satisfy $s' = w s^{2 \lambda}$ for some constant $w>0$
 and $g$ must solve the equation 
\begin{equation}\label{coag4}
w \big (-2 g(x) - x g'(x)\big) = \int_0^{x/2} \,dy\, K(\xmy,y) g(y) g(\xmy) 
\,-\, g(x) \int_0^{\infty} \,dy\,K(x,y) g(y) \,.
\end{equation}
Notice that if we have a solution $g$ of \eqref{coag4} we can get a solution
$\tilde g$
for   $w=1$ but with the same first moment $M_1$ as $g$ by 
setting
$\tilde g(x) = a^2 g(a x)$  with $a^{-1+2\lambda}=w$. Hence, we set in
the following without loss of generality $w=1$.
Furthermore,  if $g(x)$ is  a solution to \eqref{coag4}, then so is
\begin{equation}\label{coag5}
\hat g(x) = a^{1+2\lambda} g(a x) \qquad \mbox{ for } a >0\,
\end{equation}
with $M_1(\hat g) = a^{2 \lambda -1 } M_1(g)$.
The invariance \eqref{coag5} also suggests that a solution $g$ satisfies
\begin{equation}\label{coag6}
g(x) \sim h_{\lambda} x^{-(1+2 \lambda)} \qquad \mbox{ as } x \to 0 
\end{equation}
for a specific positive constant $h_{\lambda}$ that is determined by $K$ (see below).
This behaviour has  been predicted as well by physicists \cite{Le1,DE1}, but a rigorous
proof was still lacking. In \cite{EM} it has been established for kernels as in \eqref{kernel} and linear
combinations of those
that 
$g(x)x^{1+2\lambda+a} \in L^{\infty}(0,\infty)$ for any $a>0$ and
that $g(x)x^{1+2\lambda+a} \notin L^{\infty}(0,\infty)$ for any $a<0$. 
It is the main goal of this paper to improve this result. 
Let us also mention that for the diagonal kernel $K(x,y)=x^{-(1+2\lambda)} \delta(x{-}y)$
a self-similar solution with the expected power-law behavior has
been constructed in \cite{Le06}, but it is not known that every solution exhibits this behavior.

In order to proceed we have to switch to a weak formulation of \eqref{coag4}. Indeed,
the predicted singular behavior \eqref{coag6} implies that 
 both integrals on the right
hand side of \eqref{coag4} diverge. To avoid this difficulty, we 
 consider in the following a weak version of equation \eqref{coag4}.
Multiplying \eqref{coag4}
by $x$ and integrating from $x$ to $\infty$ we obtain
\begin{equation}\label{coag9}
x^2 g(x) = \int_0^x  \,dy\,g(y) \int_{x-y}^{\infty}\,dz\, y K(y,z) g(z) \,.
\end{equation}
 Indeed, the right hand side is just the mass flux at $x$. 
This weak formulation has also been essential in \cite{FL1} where the existence of a positive
fast decaying solution is established that satisfies \eqref{coag9} almost everywhere.
Later it has  been shown in \cite{CM} that 
any such solution is infinitely differentiable on $(0,\infty)$.

For the following we introduce $h$ via
\begin{equation}\label{coag8}
g(x) = x^{-(1+2\lambda)} h(x)\,
\end{equation}
such that \eqref{coag9}
becomes in terms of $h$
\begin{equation}\label{coag10}
h(x) = x^{2\lambda-1} \int_0^x\,dy\, y^{-2\lambda} h(y) \int_{x-y}^{\infty}\,dz\,
 K(y,z) z^{-(1+2\lambda)} h(z)\,.
\end{equation}

We see that \eqref{coag10} has the solution $h \equiv h_{\lambda}$, where
\[
h_{\lambda}^{-1}= \int_0^1 ds\,s^{-2\lambda} \int_{1-s}^{\infty} dt\,K(s,t) t^{-(1+2\lambda)}.
\]
Notice that due to the growth condition \eqref{kernel0b} with $\beta < 2\lambda$ this integral is well-defined.
This solution corresponds to a pure power-law solution of the 
original equation - a solution that due to its slow decay is considered unphysical.
After rescaling $h$ accordingly we consider from now on the equation
\begin{equation}\label{coag11}
h(x) = h_{\lambda}
 x^{2\lambda-1} \int_0^x \,dy\,y^{-2\lambda} h(y) \int_{x-y}^{\infty}\,dz\, K(y,z) z^{-(1+2 \lambda)} h(z)
\,
\end{equation}
that has the constant solution $h \equiv 1$.

Our main result establishes that   $h$ is uniformly bounded above and locally uniformly bounded from  below.
Thus we prove  the expected power-law behavior for small clusters of solutions to \eqref{coag9}.

\begin{theorem}\label{T.1}
Assume that $K$ satisfies \eqref{kernel0a}-\eqref{kernel0c} with $\alpha>0$ and $\lambda \in (0,1/2)$.
Let $h$
 be a positive  function that
satisfies  
\eqref{coag11} for almost  all $x \in (0,\infty)$. Then there exist positive constants $M=M(\lambda,\alpha,k_0,K_0)$ and
$m=m(\lambda,\alpha,k_0,K_0)$ such that
\begin{equation}\label{T.upperbound}
\sup_{x \in (0,\infty)} h(x) \leq M
\end{equation}
and
\begin{equation}\label{T.lowerbound}
\liminf_{x \to 0} h(x) \geq m\,.
\end{equation}
\end{theorem}

\ignore{
\begin{remark}
We do not need for our proof that $h$ is continuous on $(0,\infty)$. We just assume this to avoid to have
add the restriction 'for almost all' $x$ in the formulas. In principle we expect any weak solution to the
coagulation equation to be smooth. However, the result \cite{CM} strictly speaking does not cover the full
class of kernels considered in this paper.
\end{remark}
}
\begin{remark}
Notice that one can easily deduce from \eqref{coag10} that $\limsup_{x \to 0} h(x) \geq 1$. Of
course, we 
 expect that $\lim_{x \to 0} h(x)=1$ for any solution of \eqref{coag11} but presently a proof
is still lacking.
One main difficulty in the analysis
of \eqref{coag11} is the fact that if one linearises the coagulation operator around the expected power
law behavior one obtains in the case $\alpha=0$ terms of different homogeneity, whereas in the case $\alpha>0$ 
the homogeneity remains the same. As  also pointed out in \cite{CM,FL2} this is the main reason why the methods
developed for the case $\alpha=0$ do not apply to the case $\alpha>0$.
Furthermore, formal computations as well as numerical simulations \cite{FilL1,Lee} suggest for the case $\alpha>0$ that the
next order behavior of $h$ 
is oscillatory. This indicates that a rigorous proof of  continuity of $h$
at $x=0$ might be inherently difficult.
\end{remark}

\section{The upper bound}

In this section we will prove \eqref{T.upperbound}. The first step is to prove a uniform
bound on averages of $h$.

\begin{lemma}\label{L.average}
There exists a constant $C=C(\lambda,\alpha,k_0)$  such that
\begin{equation}\label{av1}
\sup_{R>0}  \av_{R/2}^{R}\,dx\, h(x) \leq C\,.
\end{equation}
\end{lemma}
\begin{proof}
We  integrate \eqref{coag11}  over $(aR,R)$, where $a \in [1/4,1]$ will be chosen later, to find
\begin{equation}\label{av10}
\int_{aR}^R\,dx\, h(x) \geq h_{\lambda} R^{2\lambda -1} \int_{aR}^R
\,dx\,
\int_0^x\,dy\, y^{-2\lambda} h(y) \int_{x-y}^{\infty}\,dz\,  K(y,z) z^{-(1+2 \lambda)} h(z)
\,.
\end{equation}
Now we  first switch the order of integration and  then drop
one of the resulting integrals respectively, keeping in mind
that the integrands are always nonnegative. This gives 
\[
\int_{aR}^{R} \,dx \int_0^x \,dy = \int_0^{aR}\,dy \int_{aR}^R \,dx + \int_{aR}^R \,dy \int_{y}^R
\,dx \geq \int_{aR}^R \,dy \int_{y}^R\,dx
\]
and
\[
\int_y^R\,dx \int_{x-y}^\infty \,dz = \int_0^{R-y} \,dz \int_y^{z+y} \,dx + \int_{R-y}^\infty
\,dz \int_y^R\,dx \geq \int_{R-y}^\infty
\,dz \int_y^R\,dx \,.
\]
 Using the last two inequalities
 in \eqref{av10} as well as the nonnegativity of the integrand,
the homogeneity of the kernel and \eqref{kernel0c},
we find, for any $b \in (a,1)$, that
\begin{equation}\label{av11}
\begin{split}
\int_{aR}^R&\,dx\, h(x)\\
&\geq  h_{\lambda} R^{2\lambda -1} \int_{aR}^R \,dy 
\int_{R{-}y}^{\infty} \,dz\,
(R{-}y) K(y,z) y^{-2\lambda} h(y)  z^{-(1+2 \lambda)} h(z)\\
& \geq C  R^{-2} \int_{aR}^{bR} \,dy \,(R-y) h(y) \int_{R{-}y}^R \,dz \frac{K(y,z)}{R^{2\lambda}}\,h(z)\\
& \geq C k_0 (1{-}b) R^{-1} \int_{aR}^{bR} \,dy\, h(y) \int_{R(1{-}b)}^R \,dz\, h(z)\,.
\end{split}
\end{equation}
Equation \eqref{av11} implies
\[
\av_{aR}^R \,dx\, h(x) \geq C (1{-}b) (b{-}a) \av_{aR}^{bR} \,dy\, h(y)\, \av_{R(1{-}b)}^R \,dz\, h(z)\,.
\]
Choosing now $a=1/4$ and $b=3/4$ implies $\sup_{R>0}  \tav_{R/4}^{3R/4}\,dx\, h(x) \leq C_0$, which in turn
implies the statement of the lemma.
\end{proof}

Lemma \ref{L.average} is crucial in the proof of the upper bound \eqref{T.upperbound}.
\begin{lemma}\label{L.upperbound}
There exists $M=M(\lambda,\alpha,k_0,K_0)$ such that 
\[
\sup_{x \in (0,\infty)} h(x) \leq M.
\]
\end{lemma}
\begin{proof}
Recall that the equation for $h$ is given in \eqref{coag11}.
We split the integral $\int_{x-y}^{\infty} \,dz$ into the parts
$\int_{x-y}^x\,dz $ and $\int_x^{\infty}\,dz$. The second one is the easier one and we start with
an estimate for it. In the following all constants will in general depend on the parameters
$\lambda,\alpha, k_0$ and $K_0$. 

We first claim that there
 exists a constant $C$ such that
\begin{equation}\label{upb2}
\int_{x}^{\infty}\,dz\, z^{-(1+2\lambda)} K(y,z) h(z)
\leq C x^{-2\lambda + \beta} y^{\alpha}\,.
\end{equation}
Since $y \leq z$ we have $K(y,z) \leq C y^{\alpha} z^{\beta}$. Furthermore, it follows
from \eqref{av1} that
\[
\begin{split}
\int_{x}^{\infty}\,dz\, z^{-(1+2\lambda)+\beta} h(z)
&= \sum_{n=0}^{\infty} \int_{2^nx}^{2^{n+1}x} \,dz\,  z^{-(1+2\lambda)+\beta} h(z)\\
& \leq  \sum_{n=0}^{\infty} \big( 2^nx\big)^{-(1+2\lambda)+\beta} \int_{2^nx}^{2^{n+1}x} \,dz\,   h(z)\\
&\leq C  \sum_{n=0}^{\infty} \big( 2^nx\big)^{-2\lambda+\beta} \\
& \leq C x^{-2\lambda + \beta}\,
\end{split}
\]
and this implies \eqref{upb2}.

Furthermore, we have  that
\begin{equation}\label{upb3}
\int_0^{x} \,dy\, y^{-2\lambda + \alpha} h(y)
\leq C x^{1-2\lambda + \alpha}\,.
\end{equation}
Indeed, we can estimate, using \eqref{av1},
\[
\begin{split}
\int_0^{x} \,dy\, y^{-2\lambda + \alpha} h(y)& = \sum_{n=0}^{\infty}
\int_{2^{-(n+1)}x}^{2^{-n} x} \,dy\, y^{-2\lambda + \alpha} h(y)\\
& \leq  \sum_{n=0}^{\infty}
\Big( 2^{-(n+1)}  x  \Big)^{-2\lambda + \alpha}
\int_{2^{-(n+1)}x}^{2^{-n} x} \,dy\, h(y)\\
& \leq C \sum_{n=0}^{\infty} 2^{-(n+1)(1-2\lambda + \alpha)} {x}^{1-2\lambda + \alpha}\\
&= C x^{1-2\lambda + \alpha}\,.
\end{split}
\]

 Combining now (\ref{upb2}) and (\ref{upb3}) we find
\begin{equation}\label{upb10}
x^{2\lambda-1} \int_0^x \,dy\, y^{-2\lambda} h(y) \int_{x}^{\infty} \,dz\, z^{-(1+2\lambda)} K(y,z) h(z) \leq C\,.
\end{equation}

To estimate the integrals $\int_0^x \int_{x{-}y}^x \cdots$ we just use the estimate \eqref{kernel0b} for $K$.
In the following we show how to estimate the term coming from $y^{\alpha} z^{\beta}$. The estimate of the second
term follows analogously.

We claim that there exists a constant $C$ such that 
\begin{equation}
\label{upb1}
\int_{\xmy}^x \,dz\, z^{-(1+2\lambda)+\beta} h(z) \leq C \Big( (\xmy)^{-2\lambda + \beta}
+ x^{-2\lambda+\beta}\Big)\,.
\end{equation}
In fact, given $x$ and $\xmy$ we define $n_0 \in \N$ such that
$2^{-(n_0+1)} x \leq \xmy \leq 2^{-n_0} x$ and split
\begin{equation}\label{upb1b}
\begin{split}
\int_{\xmy}^x \,dz\, z^{-(1+2\lambda)+\beta} h(z) &
= \int_{\xmy}^{2^{-n_0} x} \cdots + \sum_{n=0}^{n_0-1} \int_{2^{-(n+1)}x}^{2^{-n}x}
\cdots \\
& \leq (\xmy)^{-(1+2\lambda)+\beta} \int_{\xmy}^{2^{-n_0} x}\,dz\, h(z) \\
&\quad \,+\, 
\sum_{n=0}^{n_0-1} \big( 2^{-(n+1)} x\big)^{-(1+2\lambda)+\beta} \int_{2^{-(n+1)}x}^{2^{-n}x}
\,dz\,h(z)\,.
\end{split}
\end{equation}
Due to \eqref{av1} and the definition of $n_0$ 
we have
\[
\int_{\xmy}^{2^{-n_0}x} \,dz\,h(z) \leq C 2^{-(n_0+1)}x \leq C (\xmy)\,.
\]
Using  
 \eqref{av1} also in the second term on the right hand side of \eqref{upb1b}
we find
\[
\begin{split}
\int_{\xmy}^x \,dz\, z^{-(1+2\lambda)+\beta} h(z) &
\leq C (\xmy)^{-2\lambda+\beta}  \,+\, C \sum_{n=0}^{n_0-1} \big( 2^{-(n+1)} x\big)^{-2\lambda + \beta}\\
& \leq  C \Big((\xmy)^{-2\lambda+\beta}  \,+\,  x^{-2\lambda+\beta}\Big)\,,
\end{split}
\]
which proves \eqref{upb1}.

Now
\begin{equation}\label{upb4}
\begin{split}
\int_0^x &\,dy\,y^{-2\lambda + \alpha} h(y) \Big( (\xmy)^{-2\lambda+\beta} + x^{-2\lambda+ \beta}\Big) 
\\&
\leq C \int_{x/2}^x \,dy\, y^{-2\lambda + \alpha} (\xmy)^{-2\lambda+\beta} h(y)
\,+\, C x^{-2\lambda + \beta} \int_0^{x/2} \,dy\,y^{-2\lambda + \alpha} h(y)
\,.
\end{split}
\end{equation}
By \eqref{upb3} and $\alpha+\beta=2\lambda$  we have
\begin{equation}\label{upb5}
   x^{-2\lambda + \beta} \int_0^{x/2} \,dy\,y^{-2\lambda + \alpha} h(y)
\leq C x^{1-2\lambda}\,.
\end{equation}
Finally, similarly as before,
\begin{equation}\label{upb6}
\begin{split}
\int_{x/2}^x \,dy\,& y^{-2\lambda + \alpha} (\xmy)^{-2\lambda+\beta} h(y)
\\
&\leq C x^{-2\lambda + \alpha} \int_{x/2}^x \,dy\,(\xmy)^{-2\lambda+\beta} h(y)
\\
& \leq C x^{-2\lambda + \alpha}\sum_{n=1}^{\infty} 
\int_{x-2^{-n}x}^{x- 2^{-(n+1)}x} \,dy\,(\xmy)^{-2\lambda + \beta} h(y)\\
&\leq C x^{-2\lambda + \alpha}\sum_{n=1}^{\infty}
\Big( 2^{-n} x\Big)^{-2\lambda + \beta} \int_{x-2^{-n}x}^{x- 2^{-(n+1)}x} \,dy\, h(y)\\
& \leq C x^{1-2\lambda } \sum_{n=1}^{\infty}
\Big(2^{-n}\Big)^{1-2\lambda+\beta} \\
& \leq C x^{1-2\lambda } \,.
\end{split}
\end{equation}

Thus, estimates \eqref{upb1} and \eqref{upb6} 
imply
\begin{equation}\label{upb11}
x^{2\lambda-1} \int_0^x \,dy\, y^{-2\lambda} h(y) \int_{\xmy}^{x} \,dz z^{-(1+2\lambda)} K(y,z) h(z) \leq C,
\end{equation}
which together with \eqref{upb10}  finishes the proof of the upper bound.
\end{proof}

\section{The lower bound}

For the proof of a lower bound on $\liminf_{x \to 0} h(x)$ it is convenient
 to introduce the change of variables
\begin{equation}\label{changevariables}
x=e^{\x}, \quad y=e^{\y}, \quad z=e^{\z} \mbox{ and } \quad \h(\x) = h(x)\,.
\end{equation}
Then \eqref{coag11} becomes
\begin{equation}\label{Heq}
\begin{split}
&\h(\x)\\
&= h_{\lambda}  \int_{-\infty}^0 \,d\y\,e^{(1-2\lambda )\y} 
\int_{\log(1-e^{\y})}^\infty \,d\z\,
e^{-2\lambda\z} K(e^{\x},e^{\y}) \h(\x{+}\y) \h(\x{+}\z)
\\
&= \int_{{\Omega}_{0}}\,d\y\,d\z\, G(\y,\z) \h(\x{+}\y) \h(\x{+}\z)\\
&= \int_{\Omega_{\x}} \,d\y\,d\z\,G(\y{-}\x,\z{-}\x)\h(\y)\h(\z)\,,
\end{split}
\end{equation}
with
\begin{equation}\label{omegaxdef}
\Omega_{\x} = \Big\{ -\infty < \y<\x\; ; \; \z-\x > \log \big( 1- e^{\y{-}\x}\big) \Big\}
\end{equation}
and
\begin{equation}\label{gdef}
G(\y,\z) = h_{\lambda} e^{(1-2\lambda)\y}e^{-2\lambda \z} K(e^{\y},e^{\z})\,.
\end{equation}
For further use we notice that the smoothness and homogeneity of the kernel $K$ imply
that $G(\y{-}\eps,\z{-}\eps)$ is strictly decreasing in $\eps$. Indeed, this follows from
\[
\begin{split}
\frac{d}{d\eps} G(\y{-}\eps,\z{-}\eps)&= -\partial_{\y} G(\y{-}\eps,\z{-}\eps) - \partial_{\z}
G(\y{-}\eps,\z{-}\eps)\\
& =G(\y{-}\eps,\z{-}\eps)\Big( -(1{-}2\lambda)
\\
& \qquad
  - e^{\y{-}\eps}\frac{\partial_y K}{K}(\y{-}\eps,\z{-}\eps) + 2\lambda - e^{\z{-}\eps} \frac{\partial_z K}{K}(\y{-}\eps,\z{-}\eps)\Big)\\
&= - G(\y{-}\eps,\z{-}\eps)(1{-}2\lambda) <0\,
\end{split}
\]
and, more precisely, this implies
\begin{equation}\label{gdec}
G(\y{-}\eps,\z{-}\eps) = G(\y,\z) e^{-(1{-}2\lambda) \eps}\,.
\end{equation}

\subsection{A growth estimate}

We first prove an estimate that shows that $\h$ can change at most exponentially.

\begin{lemma}\label{L.growth}
There exists a positive constant $D=D(\lambda,\alpha,k_0,K_0)$ such that for any $\x_0\in \R$ we have
\begin{equation}\label{growth}
\h(\x) \leq 2 \h(\x_0) e^{D(\x{-}\x_0)} \qquad \mbox{ for all } \x>\x_0\,.
\end{equation}
\end{lemma}
\begin{proof}
For positive $\eps>0$ we consider $\h(\x{+}\eps)$. For that purpose we write
\[
\begin{split}
\Omega_{\x{+}\eps}& = \big( \Omega_{\x+\eps} \cap \{ \y \leq \x\}\big) 
\cup \big ( \Omega_{\x{+}\eps} \cap  \{ \x < \y < \x{+} \eps\} \big)\\
& \subset \Omega_{\x{+}\eps} \cup \big ( \Omega_{\x{+}\eps} \cap  \{ \x < \y < \x+ \eps\} \big)
=: \Omega_{\x{+}\eps}  \cup \, {\tilde \Omega}_{\eps}\,.
\end{split}
\]
In the domain $\Omega_{\x}$ we have that 
$G(\y{-}\x,\z{-}\x)$
is decreasing in $\x$. Hence 
\[
\begin{split}
&\h(\x{+}\eps) \leq \int_{\Omega_{\x}}\,d\y\,d\z\, G(\y{-}(\x{+}\eps),\z{-}(\x{+}\eps)) \h(\y)\h(\z)
\\
&
\qquad\qquad \,+\, \int_{{\tilde \Omega}_{\eps}} \,d\y\,d\z\, G(\y{-}(\x{+}\eps),\z{-}(\x{+}\eps)) \h(\y)\h(\z)
\\
& \leq \h(\x) + \int_{-\eps < \y <0} \int_{\z>\log(1{-}e^{\y})}  \,d\y\,d\z\,G(\y,\z) 
\h(\y{+}\x{+}\eps) \h(\z{+}\x{+}\eps)\\
& \leq \h(\x) + M \sup_{\y\in(\x,\x+\eps)}\h(\y)  
\int_{-\eps < \y <0} \int_{\z>\log(1{-}e^{\y})}  \,d\y\,d\z\,G(\y,\z)
\,,
\end{split}
\]
where $M$ is the uniform bound from Lemma \ref{L.upperbound}.
Recall, that \eqref{kernel0b} implies for $G$ that
\begin{equation}\label{Gbound}
G(\y,\z) \leq h_{\lambda} K_0 \Big [e^{(1-2\lambda + \alpha)\y} e^{(-2\lambda+\beta)\z}
+ e^{(1-2\lambda + \beta)\y} e^{(-2\lambda+\alpha)\z}\Big]\,.
\end{equation}
We find
\[
\begin{split}
\int_{-\eps < \y <0}\,d\y\, \int_{\z>\log(1{-}e^{\y})}\,d\z 
e^{(\alpha-2\lambda)\z} e^{(1-2\lambda + \beta)\y} 
&\leq C \int_{-\eps<\y<0}\,d\y\, \y^{\alpha-2\lambda} 
\\
&\leq C \eps^{1-2\lambda+\alpha}
\end{split}
\]
and a similar term from the first part of the right hand side of \eqref{Gbound}. Since we assume that $\alpha \leq \beta$ this
gives together with the previous estimate
\[
\h(\x{+}\eps) \leq \h(\x) + C \sup_{\y\in(\x,\x+\eps)}\h(\y)  
\eps^{1-2\lambda+\alpha} 
\]
and hence

\[
\h(\x+\eps) \leq  \h(\x) + \tfrac 1 2 \sup_{\y\in(\x,\x+\eps)}\h(\y) 
\]
for sufficiently small $\eps$. Since we can obtain analogously the estimate
$H(\x+\tilde \eps) \leq  H(\x) + \frac 1 2 \sup_{\y\in(\x,\x+\eps)}\h(\y) 
$ for all $\tilde \eps \in (0,\eps)$, and thus, taking the supremum over
$\tilde \eps$, we find
\[
\sup_{\y\in(\x,\x+\eps)}\h(\y)  
\leq 2 {\h(\x)}\,.
\]
This implies the statement of the lemma.

\end{proof}

\subsection{A stability result}

Our lower bound will be a consequence of the following lemma.

\begin{lemma}\label{L.continuity}
Let $\eps \in (0,\eps_0]$ and let $\eps_0=\eps_0(\lambda,\alpha,k_0,K_0)$ be sufficiently small. Then
 there exist $L=L(\eps,\lambda,\alpha,k_0,K_0)$ and $\delta_0=\delta_0(\eps,\lambda,\alpha,k_0,K_0)$
such that the following holds true for all $\delta \in (0,\delta_0]$ and $\x_0 \in \R$.

If $\h(\x) \leq 4 \h(\x_0)$ in $[\x_0{-}L,\x_0]$ and $\h(\x_0) \leq \delta$, then
$\h(\x_0+\eps) \leq \big( 1- (1-2\lambda)\eps/4\big) \h(\x_0)$. Furthermore
$\h(\x) \leq 4 \delta$ for all $\x>\x_0$.
\end{lemma}

\begin{proof}
As in the previous lemma we have, exploiting in addition \eqref{gdec},  that
\begin{equation}\label{hest0}
\begin{split}
\h(\x_0{+}\eps)& \leq \int_{\Omega_{\x}}\,d\y\,d\z\, G(\y{-}(\x_0{+}\eps),\z{-}(\x_0{+}\eps)) \h(\y)\h(\z)
\\
&
\qquad\,+\, \int_{{\tilde \Omega}_{\eps}} \,d\y\,d\z\, G(\y{-}(\x_0{+}\eps),\z{-}(\x_0{+}\eps)) \h(\y)\h(\z)
\\
&\leq e^{-(1-2\lambda)\eps} \h(\x) +\, \int_{{\tilde \Omega}_{\eps}} \,d\y\,d\z\, G(\y{-}(\x_0{+}\eps),
\z{-}(\x_0{+}\eps)) \h(\y)\h(\z)
\,.
\end{split}
\end{equation}
We recall that Lemma \ref{L.growth} implies that
\begin{equation}\label{hest1}
\h(\x) \leq 2 e^{DL} \h(\x_0) \qquad \mbox{ for } \x \in (\x_0,\x_0+L)
\end{equation}
and in particular for sufficiently small $\eps$
\begin{equation}\label{hest2}
\h(\y) \leq 4 \h(\x_0) \qquad \mbox{ for } \y \in (\x_0,\x_0+\eps)\,.
\end{equation}
Thus, 
\eqref{hest0} implies
\begin{equation}\label{hest3}
\begin{split}
\h(\x_0{+}\eps) \leq& e^{-(1-2\lambda)\eps} \h(\x_0)\\
& + C \h(\x_0) \int_{{\tilde \Omega}_{\eps}} \,d\y\,d\z\, G(\y{-}(\x_0{+}\eps),
\z{-}(\x_0{+}\eps)) \h(\z)\,.
\end{split}
\end{equation}

We recall that ${\tilde \Omega}_{\eps} \subset \{ (\y,\z)\,:\, \x_0 <\y \leq \x_0+\eps\,, \,
\z \geq \x_1 \}$ 
with
$\x_1:= \x_0{+}\eps +  \log
( 1{-}e^{\y{-}(\x_0{+}\eps)})$. Thus
 from (\ref{Gbound})
\begin{equation}\label{hest4}
\begin{split}
\int_{{\tilde \Omega}_{\eps}} &\,d\y\,d\z\, G(\y{-}(\x_0{+}\eps),
\z{-}(\x_0{+}\eps)) \h(\z) \\
&\leq C \int_{\x_0}^{\x_0+\eps}\,d\y \int_{\x_1}^{\infty}
\,d\z\,\big[
e^{(\beta - 2 \lambda)(\z{-}(\x_0{+}\eps))}+ e^{(\alpha - 2\lambda) (\z-(\x_0+\eps))} \big]
\h(\z)\,.
\end{split}
\end{equation}
Now we   split
\[
\begin{split}
(\x_1, \infty )
& =  (\x_1,\max(\x_0{-}L,\x_1)) \cup  (\max(\x_0{-}L,\x_1),\x_0) \\
& \quad \cup (\x_0,\x_0+L) \cup
(\x_0+L,\infty)
\end{split}
\]
We will see that the integral over the third interval will be controlled by \eqref{hest1},
the last by the decay of the kernel, the second by the smallness assumption  for $\h$ 
 on $[\x_0{-}L,\x_0]$
and the first again by the property of the kernel.

Indeed, 
 using \eqref{hest1} as well as $\beta<2\lambda$, we find
\begin{equation}\label{hest5}
\int_{\x_0}^{\x_0+L} \,d\z\,\big[
e^{(\beta - 2 \lambda)(\z{-}(\x_0{+}\eps))}+ e^{(\alpha - 2\lambda) (\z-(\x_0+\eps))} \big]
\h(\z)\,
\leq C  e^{DL}\, \h(\x_0)\,.
\end{equation}
Furthermore, recalling $\alpha \leq \beta<2\lambda$, we have
\begin{equation}\label{hest6}
\int_{\x_0+L}^{\infty} 
 \,d\z\,\big[
e^{(\beta - 2 \lambda)(\z{-}(\x_0{+}\eps))}+ e^{(\alpha - 2\lambda) (\z-(\x_0+\eps))} \big]
\h(\z)\,
\leq C e^{(\beta-2\lambda)L}\,.
\end{equation}
The assumptions in the Lemma imply that
\begin{equation}\label{hest7}
\begin{split}
\int_{\max(\x_0{-}L,\x_1)}^{\x_0} \,d\z\,\big[
e^{(\beta - 2 \lambda)(\z{-}(\x_0{+}\eps))}&+ e^{(\alpha - 2\lambda) (\z-(\x_0+\eps))} \big]
\h(\z)\,\\
& \leq C \h(\x_0) e^{(2\lambda - \alpha)L}\,.
\end{split}
\end{equation}
Finally, we consider the interval $(\x_1,\max(\x_0-L,\x_1))$. This is only nonempty 
if $\y \geq \x_0+\eps + \log \big( 1 - e^{-(L+\eps)}\big)$. Using
the global bound on $\h$ from Lemma \ref{L.upperbound}, we find
\begin{equation}\label{hest8}
\begin{split}
\int_{\x_1}^{\max(\x_0-L,\x_1)}& \,d\z\,
e^{(\alpha - 2\lambda) (\z-(\x_0+\eps))} 
\h(\z)\, \\
&\leq C \int_{\log(1-e^{\y-(\x_0+\eps)})}^{-(L+\eps)} \,d\z\, e^{(\alpha - 2\lambda)\z}
\\
&\leq C \exp \big( (\alpha - 2 \lambda) \log (1-e^{\y-(\x_0+\eps)})\big)\\
& \leq C \big( 1 - e^{\y-(\x_0+\eps)})^{\alpha - 2 \lambda}
\end{split}
\end{equation}
and hence
\begin{equation}\label{hest9}
\begin{split}
\int_{\x_0{+}\eps+\log(1-e^{-(L+\eps)})}^{\x_0{+}\eps} &\,d\y\, \int_{\x_1}^{\max(\x_0-L,\x_1)}
\,d\z\,
e^{(\alpha - 2\lambda) (\z-(\x_0+\eps))}
\h(\z)\, \\
&\leq C \int_{\log(1-e^{-(L+\eps)})}^0 \,d\y\,\big( 1-e^{\y-(\x_0{+}\eps)}\big)^{\alpha-2\lambda}
\\
&\leq C \int_0^{e^{-(L{+}\eps)}}\,d\z\,\z^{\alpha-2\lambda}\\
& \leq C e^{- (1+\alpha - 2\lambda)L}\,.
\end{split}
\end{equation}
Thus we deduce from \eqref{hest4}-\eqref{hest9} that
\begin{equation}\label{hest10}
\begin{split}
\int_{{\tilde \Omega}_{\eps}} &\,d\y\,d\z\, G(\y{-}(\x_0{+}\eps),
\z{-}(\x_0{+}\eps)) \h(\z) \\
&\leq C \Big( \eps \h(\x_0) \big( e^{DL} + e^{(2\lambda-\alpha)L}\big) + \eps
e^{-(2\lambda-\beta)L} + e^{-(1+\alpha-2\lambda)L} \Big)\,.
\end{split}
\end{equation}
Plugging \eqref{hest10} into \eqref{hest3} implies 
\begin{equation}\label{hest11}
\h(\x_0{+}\eps) \leq \h(\x_0) \Big( 1 - \frac{1-2\lambda}{2}\eps + C \big( \delta \eps e^{\gamma}L
+ e^{-\sigma L}\big) \Big)
\end{equation}
with $\gamma=\max(D,2\lambda-\alpha)$ and $\sigma = \min(1+\alpha-2\lambda, 2\lambda - \beta)
= 2 \lambda - \beta$.

In all these computations we have assumed that $\eps$ is sufficiently small. Given now such an $\eps$ we 
choose $L$ sufficiently large such that $Ce^{-\sigma L} \leq \frac 1 8 (1-2\lambda)\eps$ and
then $\delta$ sufficiently small such that $C\delta \eps e^{\gamma L} \leq \frac 1 8 (1-2\lambda)\eps$ .
Then
\begin{equation}\label{hest12}
\h(\x_0{+}\eps) \leq \Big( 1 - \tfrac 1 4 (1-2\lambda)\eps \Big) \h(\x_0) \leq \h(\x_0) \leq \delta\,.
\end{equation}
Furthermore, due to \eqref{hest2} we also have $\h(\x) \leq 4 \delta $ in $(\x_0,\x_0{+}\eps)$.
Hence, the assumptions of the Lemma are satisfied for $\x_0{+}\eps$ as well. This implies the
desired result.
\end{proof}

\subsection{Consequences}

We can now easily derive the following consequences of \linebreak Lemma \ref{L.continuity}. 
\begin{lemma} \label{C.1}
There exist positive constants $L=L(\lambda,\alpha,k_0,K_0)$ and $\delta_0 =$ \linebreak $ \delta_0(\lambda,\alpha,k_0,K_0)$ such that  
for all $\delta \in (0,\delta_0]$ the following holds true.
If $\h(\x) \leq 4 \delta$ in an interval $[\x_0,\x_0+L]$ and $\h(\x_0+L)\leq \delta$
then $\h(\x)\leq 4 \delta $ for all $\x\geq\x_0+L$.
\end{lemma}

\begin{lemma}\label{C.lowerbound}
We have
\begin{equation}\label{hest20}
\liminf_{\x \to -\infty} \h(\x) >0\,.
\end{equation}
\end{lemma}
\begin{proof}
Assume that \eqref{hest20} is not satisfied. Then there exist sequences $(\x_n)$ and $(\delta_n)$ 
with $\x_n \to - \infty$ and $\delta_n \to 0$ as $n \to \infty$, such that $\h(\x_n) \leq \delta_n$. 
By Lemma \ref{L.growth} we have $\h(\x) \leq 2 e^{DL}\delta_n$ in $[\x_n,\x_n+L]$. Then, by 
Lemma \ref{C.1}, we have $\h(\x)\leq { 8} e^{DL} \delta_n$ for all $\x\geq \x_n+L$. Thus, 
$\h \equiv 0$ which gives a contradiction.
\end{proof}

\bibliographystyle{amsplain}

\end{document}